\documentclass{article}
\input cyracc.def

\usepackage{amsmath}
\usepackage{amssymb}

\usepackage{amsfonts}
\usepackage{eucal}
\usepackage{url}
\usepackage{bm}
\usepackage{hyperref}
\usepackage{mathrsfs}
\usepackage{graphicx}
\usepackage{amscd}
\usepackage{amsthm}
\usepackage{mathtools}
\usepackage{stmaryrd}
\usepackage{MnSymbol}
\usepackage[all]{xy}
\usepackage{tikz}
\usetikzlibrary{cd}
\usetikzlibrary{patterns}

\DeclareMathOperator{\gr}{gr}

\DeclareMathOperator{\op}{op}

\DeclareMathOperator{\Spec}{Spec}
\DeclareMathOperator{\Sym}{Sym}

\newcommand{\bQ}{\mathbb{Q}}
\newcommand{\bR}{\mathbb{R}}
\newcommand{\bZ}{\mathbb{Z}}
\newcommand{\cA}{\mathcal{A}}
\newcommand{\cB}{\mathcal{B}}

\newcommand{\cF}{\mathcal{F}}

\newcommand{\cM}{\mathcal{M}}
\newcommand{\cN}{\mathcal{N}}

\newcommand{\cO}{\mathcal{O}}

\newcommand{\cU}{\mathcal{U}}
\newcommand{\cV}{\mathcal{V}}
\newcommand{\cW}{\mathcal{W}}

\newcommand{\fg}{\mathfrak{g}}

\theoremstyle{plain}
\newtheorem{thm}{Theorem}[section]

\newtheorem{cor}[thm]{Corollary}
\newtheorem{lem}[thm]{Lemma}
\newtheorem{prop}[thm]{Proposition}
\newtheorem{var}[thm]{Variant}
\theoremstyle{definition}

\newtheorem{cond}[thm]{Condition}

\newtheorem{rem}[thm]{Remark}
\begin{document}
	\title{Filtrations on the globalization of twisted D-modules over Dedekind schemes}
	\author{Takuma Hayashi\thanks{Department of Pure and Applied Mathematics, Graduate School of Information Science and Technology, Osaka University, 1-5 Yamadaoka, Suita, Osaka 565-0871, Japan, hayashi-t@ist.osaka-u.ac.jp}}
	\date{}
	\maketitle
	\begin{abstract}
		In \cite{hayashijanuszewski}, we established the theory of twisted D-modules over general base schemes. In this short note, we construct a $K$-invariant positive exhaustive filtration on the globalization of the twisted D-module on a smooth quasi-compact $K$-scheme over a Dedekind scheme $S$ obtained by the direct image of a $K$-equivariant twisted integrable connection along a $K$-equivariant closed immersion from a smooth proper $K$-scheme $Y$ with $K$ a smooth $S$-affine group scheme, whose $p$th associated graded $\cO_S$-module is locally free of finite rank for every integer $p$. In particular, the $k$-module of its global sections is projective if $S$ is affine with coordinate ring $k$.
	\end{abstract}

	\section{Introduction}\label{sec:intro}
	In the 2010s, arithmetic structures of Harish-Chandra modules attract attention to some authors for applications to the study of special values of automorphic $L$-functions. In fact, the cohomology of ad\'elic locally symmetric spaces is expected to carry information about special $L$-values. The origin of this idea goes back to \cite{MR0120372}. One may wish that arithmetic structures of the decomposition of the cuspidal cohomology into cohomology of irreducible cuspidal automorphic representations, together with more arguments, for example, the study of the Eisenstein cohomology will lead to rationality or integrality results of special $L$-values. 
	For this reason, one can think of the study of arithmetic structures of Harish-Chandra modules as an Archimedean part of the theory of automorphic representations in this context.

	To get rational models of discrete series representations, M.~Harris proposed in \cite{MR3053412} and \cite{MR4073199} to use the localization. F.~Januszewski introduced the cohomological induction functor over fields of characteristic zero to get rational models of cohomologically induced modules in \cite[Sections 3 and 7]{MR3770183}. Then the author constructed the cohomological induction functor over general commutative rings in \cite{MR3853058} and \cite{MR4007195} to get models over Dedekind domains.
	
	In \cite{MR3770183}, Januszewski also proposed to study the descent problem of the fields of definition of Harish-Chandra modules as a new aspect of representation theory of real reductive Lie groups. In \cite{MR3937337}, it is shown that rationality patterns one finds for the fields of definitions of certain classes of $(\fg,K)$-modules, matches the rationality pattern predicted by Deligne's Conjecture on special values of $L$-functions. A crucial ingredient is the Galois descent of $(\fg,K)$-modules (\cite[Section 7.4]{MR3770183}).
	
	In \cite{hayashijanuszewski}, Januszewski and the author improved Harris' idea to construct smaller arithmetic models of cohomologically induced modules. In fact, cohomologically induced modules are realized as the spaces of global sections of the twisted D-module theoretic direct image of equivariant line bundles on the corresponding closed $K$-orbits to partial flag varieties (see \cite[Theorem 4.3]{MR910203} and \cite[Theorem 5.1, Corollary 5.5]{MR2945222}). Therefore we should get smaller models of cohomologically induced modules from smaller models of their geometric materials via the theory of twisted D-modules. To follow this idea, we developed the theory of twisted D-modules over general base schemes in \cite[Sections 1--4]{hayashijanuszewski}. We also discussed the descent problem of base schemes of certain closed orbits on partial flag schemes in \cite[Section 5]{hayashijanuszewski}. The author studied a cohomological classification of data of Galois descent for the rings of definition of equivariant line bundles on partial flag schemes in \cite{MR4627704}. As a consequence, we achieved the construction of arithmetic models of cohomologically induced modules and the descent of their rings of definition simultaneously. However, it was nontrivial yet that the models over Dedekind domains of dimension one are essentially new. In fact, rational models are already integral models due to the isomorphism $\bQ\otimes_{\bZ}\bQ\cong\bQ$. Here $\bZ$ and $\bQ$ are the ring of integers and the field of rational numbers respectively. In this paper, we distinguish the models of cohomologically induced modules over Dedekind domains of dimension one we constructed in \cite[Section 6]{hayashijanuszewski} from the rational models by proving the following result:
	
	\begin{thm}[Corollary \ref{cor:filt}, Proposition \ref{prop:eq}, Proposition \ref{prop:rep}, Theorem \ref{thm:extendedver}]\label{mainthm}
		Let $K$ be a smooth $S$-affine group scheme over a Dedekind scheme $S$ (i.e., a Noetherian integral regular scheme of dimension $\leq 1$, see \cite[Section (7.13)]{MR4225278}), $i:Y\hookrightarrow X$ be a $K$-equivariant closed immersion of smooth $K$-schemes over $S$, and $\cA$ be a $K$-equivariant tdo on $X$. Write $x:X\to S$ for the structure morphism. Suppose that the following conditions are satisfied:
		\begin{enumerate}
			\renewcommand{\labelenumi}{(\roman{enumi})}
			\item $X$ is quasi-compact over $S$;
			\item $Y$ is proper over $S$.
		\end{enumerate}
		Let $\cV$ be a $K$-equivariant integrable right $i^\cdot\cA$-connection in the sense of \cite[Definition 2.1.3]{hayashijanuszewski}. Then there exists a natural exhaustive $K$-invariant filtration $G_\bullet i_+\cV$ on $i_+\cV$ such that the $p$th associated graded sheaf
		\[\gr^p_G x_\ast i_+\cV\coloneqq x_\ast G_p i_+\cV/x_\ast G_{p-1}i_+\cV\]
		to the induced $K$-invariant filtration $x_\ast G_\bullet i_+\cV$ on $x_\ast i_+\cV$ is a locally free $\cO_S$-module of finite rank for every integer $p$.
	\end{thm}
	
	Our strategy for the proof is as follows: For each integer $p$, let $G_p \cA_{Y\to X}$ be the left $i^\cdot\cA$-module generated by $i^\ast F_p\cA$, where $F_p\cA$ is the $p$th filter of the given filtration on $\cA$. Define filtrations
	\[\begin{array}{c}
		G_\bullet (\cV\otimes_{i^\cdot\cA}\cA_{Y\to X}),\\
		G_\bullet i_+\cV,\\
		G_\bullet x_\ast i_+\cV
	\end{array}\]
	on $\cV\otimes_{i^\cdot\cA}\cA_{Y\to X}$, $i_+\cV$, $x_\ast i_+\cV$ by
	\[\begin{array}{c}
		G_\bullet (\cV\otimes_{i^\cdot\cA}\cA_{Y\to X})= \cV\otimes_{i^\cdot\cA} G_\bullet\cA_{Y\to X}\\
		G_\bullet i_+\cV= i_\ast G_\bullet(\cV\otimes_{i^\cdot\cA}\cA_{Y\to X})\\
		G_\bullet x_\ast i_+\cV= x_\ast G_\bullet i_+\cV
	\end{array}\]
	respectively. One can prove that $G_\bullet i_+\cV$ is $K$-invariant by putting the $K$-equivariant structure on $G_\bullet i_+\cV$ in a similar way to \cite[Theorem 3.10.2]{hayashijanuszewski}. The filtration $G_\bullet i_+\cV$ on $i_+\cV$ is exhaustive since $i$ is quasi-compact. Since $X$ is quasi-compact and quasi-separated over $S$, $G_\bullet x_\ast i_+\cV$ is an exhaustive quasi-coherent $K$-invariant filtration on the representation $x_\ast i_+\cV$ of $K$ (cf.\ \cite[Theorems 3.10.2 and 3.11.1]{hayashijanuszewski}). The key idea for the algebraic properties of $G_\bullet x_\ast i_+\cV$ is to prove an isomorphism
	\[\gr^p_G(\cV\otimes_{i^\cdot\cA}\cA_{Y\to X})\cong\cV\otimes_{\cO_Y}\Sym^p_{\cO_Y} i^\ast\Theta_X/\Theta_{Y/S}\]
	for each integer $p$ as \cite[Lemma 3.3]{MR3372316} and \cite[the proof of 7.7 Proposition]{MR1082342}. One can prove by comparing $\gr^p_G x_\ast i_+\cV$ and $y_\ast \gr^p_G (\cV\otimes_{i^\cdot\cA}\cA_{Y\to X})$ that $\gr^p_G x_\ast i_+\cV$ is a torsion-free and coherent $\cO_S$-module, where $y$ is the structure morphism $Y\to S$. The assertion now follows since $S$ is Dedekind.

	One can also prove the following result along the same lines:
	
	\begin{var}[Corollary \ref{cor:filt}, Theorem \ref{thm:extendedver}]\label{var:mainthm}
		Let $i:Y\hookrightarrow X$ be a closed immersion of smooth schemes over a Noetherian integral domain $k$ which is Dedekind in the sense of \cite[Definition/Proposition B.87]{MR4225278}, and $\cA$ be a tdo on $X$. Suppose that $Y$ is proper over $k$.
		Let $\cV$ be a integrable right $i^\cdot\cA$-connection on $Y$. Then there exists a natural exhaustive $k$-linear filtration $G_\bullet \Gamma(X,i_+\cV)$ on $\Gamma(X,i_+\cV)$ such that the $p$th associated graded $k$-module $G_p \Gamma(X,i_+\cV)/G_{p-1} \Gamma(X,i_+\cV)$ is finitely generated and projective for every integer $p$.
	\end{var}
	
	We only see the global sections this time. Hence we need not to appeal to the quasi-coherence of $G_\bullet x_\ast i_+\cV$. This is why we do not assume $x$ to be quasi-compact. The exhaustivity follows since $y$ is quasi-compact. Besides these points, the assertion is proved along the same line as Theorem \ref{mainthm}.
	
	A standard argument on the Mittag-Leffler condition (Lemma \ref{lem:mittag-leffler}) implies that our models are essentially new also in the following sense:
	
	\begin{cor}\label{cor:proj}
		Let $i:Y\hookrightarrow X$ be a closed immersion of smooth schemes over a Dedekind domain $k$, and $\cA$ be a tdo on $X$. Suppose that $Y$ is proper over $k$. Let $\cV$ be a integrable right $i^\cdot\cA$-connection on $Y$. Then $\Gamma(X, i_+\cV)$ is projective as a $k$-module.
	\end{cor}
	
	Finally, we remark that one can prove all the results in this paper for left $i^\cdot\cA$-modules $\cU$ along the same lines by replacing
	$\cA_{Y\to X}$ with $\cA_{X\gets Y}$ in \cite[Section 3.8.1]{hayashijanuszewski}.

	\renewcommand{\abstractname}{Acknowledgments}
	\begin{abstract}
		The author thanks Yoshiki Oshima for helpful comments.
		
		This work was supported by JSPS KAKENHI Grant Numbers JP21J00023 and JP22KJ2045.
	\end{abstract}
	
	\section{Notation}
	
	All the filtrations in this paper are increasing and positive. That is, any filtration $F_\bullet\cM$ on an abelian group or a sheaf of abelian groups satisfies $F_p\cM\subset F_{p+1}\cM$ for every integer $p$ and $F_{-1}\cM=0$. For an object $\cM$, equipped with a filtration $F_\bullet \cM$, we denote the associated graded object by
	\[\begin{array}{cc}
		\gr_F \cM=\oplus_p \gr^p_F \cM,
		&\gr^p_F\cM=F_p\cM/F_{p-1}\cM,
	\end{array}\]
	where $p$ runs through all integers.
	
	Let $(X,\cO_X)$ be a ringed space. We refer to the unbounded derived category of $\cO_X$-modules as $D(\cO_X)$. For an integer $d$, let $D^{\leq d}(\cO_X)\subset D(\cO_X)$ denote the full subcategory consisting of complexes $\cF^\bullet$ satisfying $H^n(\cF^{\bullet})=0$ for $n>d$. Let $\cF$ be an $\cO_X$-module. For an open subset $U\subset X$, let $\Gamma(U,\cF)$ denote the abelian group of $U$-sections of $\cF$. For an integer $p$, we refer to the $p$th symmetric power of $\cF$ as $\Sym^p_{\cO_X}\cF$. Set $\Sym_{\cO_X}\cF=\oplus_p\Sym^p_{\cO_X}\cF$ as the total symmetric power of $\cF$. 
	
	For a scheme $X$, we write $\cO_X$ for its structure sheaf. 
	
	Let $S$ be a scheme, and $X$ be a smooth $S$-scheme. We denote the structure morphism $X\to S$ by the corresponding small letter $x$. Let $\Theta_{X/S}$ be the $\cO_X$-module of $S$-vector fields on $X$. Write $\omega_{X/S}$ and $\omega^\vee_{X/S}$ for the canonical bundle on $X$ and its dual.
	
	In the rest of this paper, let $S$ be a scheme, and $i:Y\hookrightarrow X$ be an immersion of smooth schemes over $S$ unless specified otherwise. Let $i_\ast$ and $i^{-1}$ be the sheaf theoretic direct image and pullback respectively. For a left $\cO_X$-module $\cF$, set $i^\ast\cF=\cO_Y\otimes_{i^{-1}\cO_X} i^{-1}\cF$. Set $\cN_{X/Y}=i^\ast\Theta_{X/S}/\Theta_{Y/S}$ as the normal sheaf.
	
	For the basic notations on tdos and twisted D-modules over schemes, we follow \cite{hayashijanuszewski}. Let $\cA$ be a tdo on $X$ (\cite[Definition 1.1.4]{hayashijanuszewski}). Write $F_\bullet\cA$ for the given filtration on $\cA$. Recall that we defined the pullback $i^\cdot\cA$ in \cite[Section 1.3.5]{hayashijanuszewski}. Let $\cA_{Y\to X}$ denote the $(i^\cdot\cA,i^{-1}\cA)$-bimodule in \cite[Lemma 3.6.5]{hayashijanuszewski}. We will fix a right $i^\cdot\cA$-module $\cV$. Write $i_+\cV=i_\ast(\cV\otimes_{i^\cdot\cA}\cA_{Y\to X})$ in this paper (even if $i$ is not a closed immersion).

	\section{Filter $G$ on $\cA_{Y\to X}$}\label{sec:filt}
	
	The aim of this section is to study the filtration on $\cA_{Y\to X}$ explained in Section \ref{sec:intro}.
	
	Regard $\cA$ as an $\cO_X$-bimodule for the multiplication from the left and right sides. Then for each integer $p$, $F_p\cA$ is a sub-$\cO_X$-bimodule of $\cA$ by definition. Since $F_\bullet\cA$ is an exhaustive filtration on $\cA$, we obtain an initial (colimit) cocone
	\[\begin{tikzcd}
		\cdots\ar[r, hook]&F_{p-1}\cA\ar[r, hook]\ar[rd, hook]
		&F_p\cA\ar[d, hook]\ar[r, hook]
		&F_{p+1}\cA\ar[r, hook]\ar[ld, hook]
		&\cdots\\
		&&\cA.
	\end{tikzcd}\]
	Apply $i^\ast$ to get an initial cocone of $(\cO_Y,i^{-1}\cO_X)$-bimodules
	\[\begin{tikzcd}
		\cdots\ar[r, hook]&i^\ast F_{p-1}\cA\ar[r, hook]\ar[rd, hook]
		&i^\ast F_p\cA\ar[d, hook]\ar[r, hook]
		&i^\ast F_{p+1}\cA\ar[r, hook]\ar[ld, hook]
		&\cdots\\
		&&\cA_{Y\to X}.
	\end{tikzcd}\]
	The horizontal arrows are monic since $\gr^p_F\cA$ are locally free of finite rank as left $\cO_X$-modules for all integers $p$. Therefore $F_\bullet\cA_{Y\to X}\coloneqq i^\ast F_\bullet\cA$ exhibits an exhaustive filtration on $\cA_{Y\to X}$ as an $(\cO_Y,i^{-1}\cO_X)$-bimodule. Since $i^\ast$ is right exact, we also have a canonical $(\cO_Y,i^{-1}\cO_X)$-bilinear isomorphism $i^\ast \gr^p_F\cA\cong \gr^p_F \cA_{Y\to X}$ for each integer $p$. Write $\phi$ for the total isomorphism $i^\ast \gr_F\cA\cong \gr_F \cA_{Y\to X}$.
	
	Let $G_p\cA_{Y\to X}$ be the left $i^\cdot\cA$-submodule of $\cA_{Y\to X}$ generated by $F_p\cA_{Y\to X}$. Since $F_p\cA$ is a right $\cO_X$-submodule of $\cA$, $G_p\cA_{Y\to X}$ is a right $i^{-1}\cO_X$-submodule of $\cA_{Y\to X}$.
	
	\begin{thm}\label{thm:gr_GAYX}
		There is a canonical graded $(i^\cdot\cA,i^{-1}\cO_X)$-bilinear isomorphism
		\[i^\cdot\cA\otimes_{\cO_Y} \Sym_{\cO_Y} \cN_{X/Y}\cong \gr_G \cA_{Y\to X},\]
		where $\Sym_{\cO_Y} \cN_{X/Y}$ is regarded as a right $i^{-1}\cO_X$-module for the multiplication through the structure homomorphism $i^{-1}\cO_X\to\cO_Y$.
	\end{thm}
	
	\begin{proof}
		Recall that we are given an $\cO_X$-algebra isomorphism \[\sigma:\Sym_{\cO_X}\Theta_{X/S}\cong \gr_F\cA\]
		by definition of tdos. Compose $i^\ast\sigma$ with the canonical isomorphism
		\[\Sym_{\cO_Y} i^\ast\Theta_{X/S}\cong i^\ast\Sym_{\cO_X}\Theta_{X/S}\]
		to get
		\[\Sym_{\cO_Y} i^\ast\Theta_{X/S}\cong\Sym_{\cO_Y} i^\ast\Theta_{X/S}\overset{i^\ast\sigma}{\cong}
		i^\ast\gr_F\cA,\]
		which we will denote by $\psi$. We regard $\psi$ as an $(\cO_Y,i^{-1}\cO_X)$-bilinear isomorphism. Here each vertex in the above sequence is a right $i^{-1}\cO_X$-module for the structure homomorphism $i^{-1}\cO_X\to \cO_Y$.
		
		Since $F_p\cA_{Y\to X}\subset G_p\cA_{Y\to X}$ for every $p$, we have a canonical $(\cO_Y,i^{-1}\cO_X)$-bilinear homomorphism $\pi:\gr_F\cA_{Y\to X}\to \gr_G\cA_{Y\to X}$. Let $\tau$ denote the composite map
		\[\Sym_{\cO_Y} i^\ast\Theta_{X/S}\overset{\psi}{\cong} i^\ast\gr_F\cA\overset{\phi}{\cong} \gr_F \cA_{Y\to X}\overset{\pi}{\to} \gr^p_G\cA_{Y\to X}.\]
		We wish to prove that $\tau$ descends to an $(\cO_Y,i^{-1}\cO_X)$-module homomorphism $\Sym_{\cO_Y}\cN_{X/Y}\to \gr^p_G\cA_{Y\to X}$. This problem is local in $X$ and $Y$ since the desired homomorphism $\Sym_{\cO_Y}\cN_{X/Y}\to \gr^p_G\cA_{Y\to X}$ exists at most uniquely. Therefore we may take $(x_i)$, $(\partial_i)$, and $r\leq n$ in \cite[Theorem A.2.20]{hayashijanuszewski}. Then one can identify $\cN_{X/Y}$ with $\oplus_{r+1\leq j\leq n} \cO_Y \partial_j\subset i^\ast\Theta_{X/S}$ as an $\cO_Y$-module. The proof is completed by showing that the composite map \[\cO_Y\left[\partial_{r+1},\partial_{r+2},\ldots,\partial_n\right]
		\subset \Sym_{\cO_Y} i^\ast\Theta_{X/S}\overset{\tau}{\to} \gr^p_G\cA_{Y\to X}\]
		exhibits the desired homomorphism. For this, it will suffice to show \[\sigma(a\partial^{l_1}_1\partial^{l_2}_2\cdots \partial^{l_n}_n)=0\]
		for a local section $a\in\cO_Y$ and nonnegative integers $l_1,l_2,\ldots,l_n$ with $\sum_{j=1}^{r} l_j>0$. Choose a lift $e_j\in F_1\cA$ of $\partial_j$ for each $1\leq j\leq n$. Then $\sigma(a\partial^{l_1}_1\partial^{l_2}_2\cdots \partial^{l_n}_n)$ lifts to $a\otimes e^{l_1}_1e^{l_2}_2\cdots el^{l_n}_n
		\in G_{\sum_{j=r+1}^n l_j} \cA_{Y\to X} \subset G_{\sum_{j=1}^n l_j-1}\cA_{Y\to X}$. This shows $\sigma(a\partial^{l_1}_1\partial^{l_2}_2\cdots \partial^{l_n}_n)=0$
		in $\gr^{\sum_{j=1}^n l_j}_G \cA_{Y\to X}$.
		
		Finally, we show the scalar extension $\tilde{\tau}:i^\cdot\cA\otimes_{\cO_Y}\Sym_{\cO_Y}\cN_{X/Y}\to \gr_G\cA_{Y\to X}$ is an isomorphism. We may again choose $(x_i)$, $(\partial_i)$, and $r\leq n$ in \cite[Theorem A.2.20]{hayashijanuszewski} and the lift $(e_j)$ as above. Then
		\[1\otimes \partial^{l_{r+1}}_{r+1} \partial^{l_{r+2}}_{r+2}
		\cdots \partial^{l_n}_n\]
		with $l_{r+1},l_{r+2},\ldots,l_{n}\geq 0$ form a free basis of $i^\cdot\cA\otimes_{\cO_Y}\Sym_{\cO_Y}\cN_{X/Y}$ as a left $i^\cdot\cA$-module. The homomorphism $\tilde{\tau}$ sends them to
		$1\otimes e^{l_{r+1}}_{r+1} e^{l_{r+2}}_{r+2}\cdots e^{l_n}_n$, which form a free basis of $\gr_G \cA_{Y\to X}$ by \cite[Proof of Lemma 3.8.10]{hayashijanuszewski}. This completes the proof.
	\end{proof}
	
	Set
	\[\begin{array}{c}
		G_\bullet (\cV\otimes_{i^\cdot\cA} \cA_{Y\to X})
		\coloneqq\cV\otimes_{i^\cdot\cA} G_\bullet\cA_{Y\to X}\\
		G_\bullet i_+\cV \coloneqq
		i_\ast(\cV\otimes_{i^\cdot\cA} G_\bullet\cA_{Y\to X})\\
		G_\bullet x_\ast i_+\cV \coloneqq 
		x_\ast G_\bullet i_+\cV
		=y_\ast G_\bullet (\cV\otimes_{i^\cdot\cA} \cA_{Y\to X}).
	\end{array}\]
	
	\begin{cor}\label{cor:filt}
		\begin{enumerate}
			\renewcommand{\labelenumi}{(\arabic{enumi})}
			\item The $i^{-1}\cO_X$-modules $G_\bullet (\cV\otimes_{i^\cdot\cA} \cA_{Y\to X})$ form an exhaustive filtration on $\cV\otimes_{i^\cdot\cA} \cA_{Y\to X}$ with
			\[\gr_G(\cV\otimes_{i^\cdot\cA} \cA_{Y\to X})\cong \cV\otimes_{\cO_Y}\Sym_{\cO_Y}\cN_{X/Y}.\]
			\item The $\cO_X$-modules $G_\bullet i_+\cV$ form a filtration on $i_+\cV$. Moreover, it is exhaustive if $i$ is quasi-compact.
			\item The $\cO_S$-modules $G_\bullet x_\ast i_+\cV$ form a filtration on $x_\ast i_+\cV$. Moreover, it is exhaustive if $y$ is quasi-compact.
		\end{enumerate}	
	\end{cor}
	
	\begin{proof}
		Part (1) is an easy consequence of Theorem \ref{thm:gr_GAYX} (cf.\ the argument at the beginning of this section). The first part of (2) (resp.\ (3)) follows since $i_\ast$ (resp.\ $y_\ast$) respects monomorphisms. The latter parts of (2) and (3) follow from the general fact that the direct image functor of sheaves along a quasi-compact morphism of schemes respects direct limits of sheaves with monic transition maps. This completes the proof.
	\end{proof}

	\begin{rem}\label{rem:directimageforleftmod}
		If you wish to work with a left $i^\cdot\cA$-module $\cU$, you tentatively define $i_+\cU=i_\ast(\cA_{X\gets Y}\otimes_{i^\cdot\cA}\cU)$. One can identify it with
		\[i_\ast((\cU\otimes_{\cO_Y}\omega_{Y/S})\otimes_{i^\cdot\cB} \cB_{Y\to X})\otimes_{\cO_X}\omega^\vee_{X/S}\]
		without any finiteness conditions by the projection formula for ringed spaces, where $\cB\coloneqq \omega^{\vee}_{X/S}\otimes_{\cO_X}\cA^{\op}\otimes_{\cO_X}\omega_{X/S}$. The sheaf $\cA^{\op}$ is the opposite tdo (\cite[Example 1.1.7]{hayashijanuszewski}).
	\end{rem}
	
	\section{Equivariant structure}
	
	In this section, we note equivariant structures on the filtrations. Let $K$ be a smooth $S$-affine group scheme over $S$, $i:Y\to X$ be a $K$-equivariant immersion of smooth schemes over $S$, and $\cA$ be a $K$-equivariant tdo on $X$. Moreover, assume:
	
	\begin{cond}\label{cond:finitedimension}
		\begin{enumerate}
			\renewcommand{\labelenumi}{(\roman{enumi})}
			\item $i$ is affine;
			\item $\bR i_\ast:D(i^{-1}\cO_X)\to D(\cO_X)$ is locally bounded in the sense of \cite[Definition A.6.6]{hayashijanuszewski}.
		\end{enumerate}
	\end{cond}
	
	Let $\cV$ be a $K$-equivariant right $i^\cdot\cA$-module. One can prove the following result along the same line as \cite[Claim 3.9.8 (1)]{hayashijanuszewski}:
	
	\begin{prop}\label{prop:xbc}
		Suppose that we are given a Cartesian diagram of smooth $S$-schemes
		\[\begin{tikzcd}
			\tilde{Y}\ar[r, "\tilde{p}"]\ar[d, "\tilde{i}"']
			&Y\ar[d, "i"]\\
			\tilde{X}\ar[r, "p"]&X.
		\end{tikzcd}\]
		Assume that the following conditions are satisfied:
		\begin{enumerate}
			\renewcommand{\labelenumi}{(\roman{enumi})}
			\item $p$ is isomorphic in the category of morphisms of $S$-schemes to a projection morphism in the category of smooth $S$-schemes.
			\item The functors
			\[\bR i_\ast:D(i^{-1}\cO_X)\to D(\cO_X)\]
			\[\bR \tilde{i}_\ast:D(\tilde{i}^{-1}\cO_{\tilde{X}})\to D(\cO_{\tilde{X}})\]
			are locally bounded.
		\end{enumerate}
		Let $p$ be an integer, and $\cV$ be a quasi-coherent right $i^\cdot\cA$-module. Then there is a natural isomorphism
		\[p^\ast G_pi_+\cV\cong G_p\tilde{i}_+ \tilde{p}^\ast\cV.\]
	\end{prop}
	
	We then obtain a $K$-equivariant structure on $G_\bullet i_+\cV$ by formal applications of Proposition \ref{prop:xbc}:
	
	\begin{prop}\label{prop:eq}
		For each integer $p$, $G_pi_+\cV$ is naturally equipped with the structure of a $K$-invariant quasi-coherent submodule of $i_+\cV$.
	\end{prop}
	
	Note that the quasi-coherence follows by a similar argument to \cite[Variant 3.8.8]{hayashijanuszewski}.
	
	The following result is a consequence of \cite[(6.7) Theorem and (5.6) Theorem]{MR337971}:
	
	\begin{prop}\label{prop:rep}
		Suppose that $x$ is quasi-compact and quasi-separated. Then for each integer $p$, $x_\ast G_pi_+\cV\subset x_\ast i_+\cV$ is naturally equipped with the structure of a subrepresentation of $K$ (cf.\ see the beginning of \cite[Section 3.11]{hayashijanuszewski} for the definition of representations of $K$).
	\end{prop}
	
	\begin{rem}\label{rem:qs}
		The structure morphism $x$ is quasi-separated if $X$ is locally Noetherian (\cite[Corollary 10.24]{MR4225278}). A typical case in our paper is when $S$ is locally Noetherian since $x$ is locally of finite type (\cite[Proposition 10.9]{MR4225278}).
	\end{rem}
	
	\section{Algebraic properties of $G_\bullet x_\ast i_+\cV$}
	
	In this section, we verify the algebraic properties of the filtered module $G_\bullet x_\ast i_+\cV$ stated in Section \ref{sec:intro}.
	
	Let $S$ be a Dedekind scheme, $i:Y\hookrightarrow X$ be a closed immersion of smooth schemes over $S$, and $\cA$ be a tdo on $X$. Let $\cV$ be a integrable right $i^\cdot\cA$-connection. Moreover, suppose that $Y$ is proper.

	\begin{thm}\label{thm:extendedver}
		\begin{enumerate}
			\renewcommand{\labelenumi}{(\arabic{enumi})}
			\item The $\cO_S$-module $\gr^p_G x_\ast i_+\cV$ is torsion-free for every integer $p$.
			\item For each integer $p$ and a nonempty affine open subscheme $U\subset S$ with coordinate ring $k$, $\Gamma(U,\gr^p_G x_\ast i_+\cV)$ is finitely generated and projective as a $k$-module.
			\item For each integer $p$, $\gr^p_G x_\ast i_+\cV$ is locally free of finite rank as an $\cO_S$-module if $x$ is quasi-compact.
		\end{enumerate}
	\end{thm}
	
	For its proof, let us note a general observation on the torsion-free property:
	
	\begin{lem}\label{lem:flat}
		Let $f:W\to Z$ be a flat, quasi-compact, and quasi-separated morphism of schemes, and $\cW$ be a flat quasi-coherent $\cO_W$-module. If $Z$ is integral then $f_\ast\cW$ is torsion-free as an $\cO_Z$-module in the sense of \cite[Chapitre I, (7.4.1)]{MR217083}.
	\end{lem}
	
	\begin{proof}
		We may assume $Z$ affine. Write $Z=\Spec R$. Since $f_\ast\cW$ is a quasi-coherent $\cO_Z$-module (\cite[(5.6) Theorem]{MR337971}), it will suffice to show that $\Gamma(Z,f_\ast \cW)=\Gamma(W,\cW)$ is a torsion-free $R$-module. Choose an affine open covering $W=\cup_\lambda U_\lambda$. Then $\Gamma(W,\cV)$ is identified with an $R$-submodule of $\prod_\lambda\Gamma(U_\lambda,\cW|_{U_\lambda})$. We may therefore assume $W$ affine. In this case, $\Gamma(W,\cV)$ is a flat $R$-module, in particular, torsion-free as an $R$-module. This completes the proof.
	\end{proof}

	\begin{proof}[Proof of Theorem \ref{thm:extendedver}]
		We first study $y_\ast \gr^p_G (\cV\otimes_{i^\cdot\cA}\cA_{Y\to X})$. Since $\cV$ and $\cN_{X/Y}$ are locally free $\cO_Y$-modules of finite rank, so is
		\[\cV\otimes_{\cO_Y}\Sym^p_{\cO_Y} \cN_{X/Y}.\]
		Hence $y_\ast \gr^p_G (\cV\otimes_{i^\cdot\cA}\cA_{Y\to X})$ is torsion-free (resp.\ coherent) by Corollary \ref{cor:filt} (1) and Lemma \ref{lem:flat} (resp.\ and \cite[Th\`eor\'eme (3.2.1)]{MR217085}).
		
		Part (1) is now obtained by regarding $\gr^p_G x_\ast i_+\cV$ as an $\cO_S$-submodule of $y_\ast \gr^p_G (\cV\otimes_{i^\cdot\cA}\cA_{Y\to X})$. For (2), let $U$ be an arbitrary nonempty affine open subscheme of $S$ with coordinate ring $k$. Then one can show in a similar way to (1) that $\Gamma(U,\gr^p_G x_\ast i_+\cV)$ is a finitely generated and torsion-free $k$-module. Since $k$ is a Dedekind domain, $\Gamma(U,\gr^p_G x_\ast i_+\cV)$ is a finitely generated and projective $k$-module (\cite[Proposition B.89 (4)]{MR4225278}). This shows (2).
		
		Finally, we prove (3). Recall that the filtration $G_\bullet x_\ast i_+\cV$ consists of quasi-coherent $\cO_S$-modules by Proposition \ref{prop:rep} (put $K=S$). The assertion now follows from (2).
	\end{proof}
	
	\begin{rem}
		If $Y$ is integral, the same argument works under the hypothesis that $\cV$ is a right $i^\cdot\cA$-module which is torsion-free and coherent as an $\cO_Y$-module by using \cite[Chapitre I, Proposition (7.4.5)]{MR217083}.
	\end{rem}

	We deduce Corollary \ref{cor:proj} by applying the following general fact:
	
	\begin{lem}\label{lem:mittag-leffler}
		Let $k$ be a commutative ring. Let $V$ be a $k$-module, equipped with an exhaustive filtration $G_\bullet V$. If the associated graded $k$-modules $\gr^p_G V$ are projective for all $p$, so is $V$.
	\end{lem}
	
	Indeed, $V=\Gamma(X,i_+\cV)$ and $G_\bullet V=\Gamma(X,G_\bullet i_+\cV)$ satisfy the conditions of Lemma \ref{lem:mittag-leffler} under the hypothesis of Corollary \ref{cor:proj} by Theorem \ref{thm:extendedver} and Corollary \ref{cor:filt} (3).

	\appendix
	\section{Compute $\gr_G i_+\cV$}\label{sec:affimm}
	
	Recall that F.~V.~Bien proved an isomorphism $\gr_G i_+\cV\cong i_\ast(\cV\otimes_{\cO_Y} \Sym_{\cO_Y} \cN_{X/Y})$ in \cite{MR1082342} if $i$ is a closed immersion. Indeed, this isomorphism holds over general bases if $i$ is a closed immersion since $i_\ast$ is then exact.
	
	In this section, assume Condition \ref{cond:finitedimension} and $\cV$ to be quasi-coherent. We aim to prove the isomorphism $\gr_G i_+\cV\cong i_\ast(\cV\otimes_{\cO_Y} \Sym_{\cO_Y} \cN_{X/Y})$ in this setting.
	
	\begin{lem}\label{lem:vanishing}
		Let $\cW$ be an $(i^\cdot\cA,i^{-1}\cO_X)$-bimodule.
		Suppose that $\cW$ is flat and quasi-coherent as a left $i^\cdot\cA$-module. Then $R^p i_\ast(\cV\otimes_{i^\cdot\cA} \cW)$ vanishes unless $p=0$.
	\end{lem}
	
	\begin{proof}
		It follows by definition that the $p$th derived functor vanishes if $p<0$. We see the assertion for $p>0$. We may assume the following conditions by the hypothesis (ii) in Condition \ref{cond:finitedimension}:
		\begin{enumerate}
			\renewcommand{\labelenumi}{(\roman{enumi})}
			\item $X$ is affine;
			\item $\bR i_\ast:D(i^{-1}\cO_X)\to D(\cO_X)$ is bounded in the sense of \cite[Definition A.6.2]{hayashijanuszewski}.
		\end{enumerate}
		One can see by taking a free resolution
		$\cF^\bullet\to \cV\to 0$ of right $i^\cdot\cA$-modules that the proof is completed by showing $R^p i_\ast(\cF^\bullet\otimes_{i^\cdot\cA} \cW)=0$ for $p>0$ and complexes $\cF^\bullet$ of free right $i^\cdot\cA$-modules concentrated in nonpositive degrees. We may and do assume $\cF^\bullet=\cF$ is concentrated in degree zero by applying stupid truncations. In this case, the assertion follows
		from \cite[Th\'eor\`eme (1.3.1)]{MR217085} since the sheaf $\cF\otimes_{i^\cdot\cA} \cW$ of abelian groups on $Y$ admits the structure of a quasi-coherent $\cO_Y$-module (see \cite[Proposition 2.2.2 and the paragraph below Proposition A.4.13]{hayashijanuszewski}.
	\end{proof}

	\begin{thm}\label{thm:computegri_+}
		There is a canonical isomorphism
		\[\gr_G i_+\cV\cong
		i_\ast(\cV\otimes_{\cO_Y} \Sym_{\cO_Y} \cN_{X/Y})\]
		of $\cO_X$-modules.
	\end{thm}
	
	\begin{proof}
		For each integer $p$, consider the short exact sequence
		\[0\to \cV\otimes_{i^\cdot\cA} G_{p-1} \cA_{Y\to X} 
		\to \cV\otimes_{i^\cdot\cA} G_{p} \cA_{Y\to X}
		\to \cV\otimes_{i^\cdot\cA} \gr^{p}_G \cA_{Y\to X}
		\to 0.\]
		It gives rise to an exact sequence
		\[\begin{split}
			0\to i_\ast(\cV\otimes_{i^\cdot\cA} G_{p-1} \cA_{Y\to X})
			\to i_\ast(\cV\otimes_{i^\cdot\cA} G_{p} \cA_{Y\to X})
			&\to i_\ast(\cV\otimes_{i^\cdot\cA} \gr^{p}_G \cA_{Y\to X})\\
			&\to R^1 i_\ast (\cV\otimes_{i^\cdot\cA} G_{p-1} \cA_{Y\to X}).
		\end{split}
		\]
		One can see from the proof of Theorem \ref{thm:gr_GAYX} that $G_{p-1}\cA_{Y\to X}$ is locally free as a left $i^\cdot\cA$-module. In particular, $G_{p-1}\cA_{Y\to X}$ is a flat and quasi-coherent left $i^\cdot\cA$-module. Hence we can apply Lemma \ref{lem:vanishing} to $\cW=G_{p-1}\cA_{Y\to X}$ to show that the first derived functor $R^1 i_\ast (\cV\otimes_{i^\cdot\cA} G_{p-1} \cA_{Y\to X})$ vanishes. The assertion now follows as
		\[\gr^p_G i_+\cV\cong 
		i_\ast(\cV\otimes_{i^\cdot\cA} \gr^{p}_G \cA_{Y\to X}) \cong
		i_\ast(\cV\otimes_{\cO_Y} \Sym^p_{\cO_Y} \cN_{X/Y})\]
		(use Theorem \ref{thm:gr_GAYX} for the last isomorphism).
	\end{proof}

\end{document}